\documentclass[english]{article}

\usepackage[T1]{fontenc}
\usepackage[utf8]{inputenc}
\usepackage{babel}

\usepackage{color}

\usepackage{graphicx}
\usepackage{amsmath}
\usepackage{amsthm}
\usepackage{amssymb}
\usepackage{amsfonts}
\usepackage{latexsym}

\definecolor{blue}{rgb}{0,0,0.8}
\definecolor{red}{rgb}{0.8,0,0}
\definecolor{darkgreen}{rgb}{0,0.6,0}

\def\beq{\begin{equation}}
\def\eeq{\end{equation}}

\newcounter{othm}
\def\theothm{\Alph{othm}}

\newcommand{\PP}{{\mathcal P}}
\newcommand{\PPP}{{\mathbb P}}

\newcommand{\B}{{\mathcal B}}
\newcommand{\zB}{\zeta_{\mathcal B}}
\newcommand{\zBs}{\zeta_{\mathcal B}(s)}

\newcommand{\lz}{\lambda_{\zeta}}
\newcommand{\lB}{\lambda_{\zB}}

\newcommand{\ee}{\eta}

\newcommand{\G}{{\mathcal G}}

\newcommand{\RR}{{\mathbb R}}

\newcommand{\NN}{{\mathbb N}}

\newcommand{\N}{{\mathcal N}}
\newcommand{\R}{{\mathcal R}}

\newcommand{\ve}{{\varepsilon}}

\newtheorem{theorem}{Theorem}
\newtheorem{corollary}{Corollary}
\newtheorem{lemma}{Lemma}

\newtheorem{definition}{Definition}

\theoremstyle{definition}

\newtheorem{remark}{Remark}

\newcommand{\de}{\delta}

\begin{document}

\title
{New zero-density estimates for the Beurling $\zeta$ function\thanks{This research was supported by Hungarian National Research, Development and Innovation Office, Project \#s KKP133819, K-147153 and
K-146387.}}

\author{J\'anos Pintz and Szil\' ard Gy. R\' ev\' esz}

\date{}

\maketitle

\begin{abstract}
In two previous papers the second author proved some Carlson type density theorems for zeroes in the critical strip for Beurling zeta functions satisfying Axiom A of Knopfmacher. In the first of these invoking two additonal conditions were needed, while in the second an explicit, fully general result was obtained. Subsequently, Frederik Broucke and Gregory Debruyne obtained, via a different method, a general Carlson type density theorem with an even better exponent, and recently Frederik Broucke improved this further, getting $N(\sigma,T) \le T^{a(1-\sigma)}$ with any $a>\dfrac{4}{1-\theta}$. Broucke employed a new mean value estimate of the Beurling zeta function, while he did not use the method of Hal\'asz and Montgomery.

Here we elaborate a new approach of the first author, using the classical zero detecting sums coupled with a kernel function technique and Halász' method, but otherwise arguing in an elementary way avoiding e.g. mean value estimates for Dirichlet polynomials. We will make essential use of the additional assumptions that the Beurling system of integers consists of natural numbers, and that the system satisfies the Ramanujan condition, too. This way we give a new variant of the Carlson type density estimate with similar strength as Turán's 1954 result for the Riemann $\zeta$ function, coming close even to the Density Hypothesis for $\sigma$ close to 1.

\end{abstract}



\medskip

\section{Zero density results for the Riemann zeta function}\label{sec:classicalintro}

In the classical case of natural numbers $\NN$ and primes $\PPP$, it is well-known from Riemann's times that the distribution of the zeroes of the (analytic continuation of) the Rieman zeta function $\zeta$ is decisive in regard of questions of prime distribution. As all the nontrivial zeroes are in the critical strip $0<\Re s <1$ and their number up to height $T$ is asymptotically $N(T) \sim \frac{T}{2\pi} \log T$, the estimations of the number $N(\sigma,T)$ of zeroes in a halfplane $\Re s >\sigma$ and up to height $T$ by quantities of smaller order could point towards validity of the Riemann Hypothesis. We are discussing here such so-called \emph{density estimates} first for the Riemann zeta function itself, and then, more generally, we use some of these results to obtain new density estimates for Beurling number systems as well.

Bohr and Landau \cite{BL1914} proved in 1914 that $N(\sigma,T)=o(N(T))$ if $\sigma>1/2$. This was the first estimate of the kind what we call density estimates.

A few years later Carlson \cite{Carlson} proved
\begin{equation}
\label{eq:1.2}
N(\sigma, T) \ll_\varepsilon T^{A(\sigma)(1 - \sigma) + \varepsilon } \ \text{ for any } \ \varepsilon  > 0 \ \text{ and } \ \sigma \geq 1/2,
\end{equation}
where he could put
\begin{equation}
\label{eq:1.3}
A(\sigma) \leq 4\sigma.
\end{equation}

Number theoretic consequences in connection with the Landau problem on the difference of consecutive primes made it of top interest to improve these bounds as much as possible. Landau formulated the conjecture that there is a prime between any two squares \cite{Lan1913}. In particular, it was shown that an almost exact positive answer to Landau's question would follow from the so-called "\emph{Density Hypothesis}" (DH) which asserts
\begin{equation}
\label{eq:1.8}
N(\sigma, T) \ll T^{2(1 - \sigma)} \log^C T \ \text{ with some }\ C > 0 \ \text{ for all } \sigma \geq \frac12,
\end{equation}
or, in a slightly weaker form, using the notation \eqref{eq:1.2},
\begin{equation}
\label{eq:1.9}
A(\sigma) \leq 2 \ \text{ for all } \ \sigma \geq 1/2 \ \ \left(\Longleftrightarrow N(\sigma, T) \ll_\varepsilon  T^{2(1 - \sigma) + \varepsilon }\right).
\end{equation}

Clearly, the Riemann Hypothesis contains the Density Hypothesis. However, Ingham \cite{Ing1937} could demonstrate DH from another, wekaer assumption, the so-called "\emph{Lindel\"of Hypothesis}" (LH), which formulates
\begin{equation}
\label{eq:1.10}
\mu\left(\frac12\right) = 0, \ \text{ where }\ \mu(\alpha) := \text{\rm inf} \left\{\mu; |\zeta(\sigma + it)| \leq T^\mu \text{ for } \sigma \geq \alpha, \, 1 < |t| \leq T\right\} .
\end{equation}

In 1954 Tur\'an \cite{Tur1954} used his celebrated power-sum method \cite{Tur1953}, see also \cite{Tur1984}, to give a different proof of Ingham's result that LH implies DH, and he almost achieved the DH in the vicinity of the boundary line $\sigma = 1$. Let us put
\begin{equation}
\label{eq:1.11}
s = \sigma + it, \ \ \eta:= 1-\sigma, \ \ B(\eta):=\frac{1}{\eta} A(1 - \eta) .
\end{equation}
Then Turán's result was with a small positive constant $c_1$ \cite{Tur1954}
\begin{equation}
\label{eq:1.12}
N(1 - \eta, T) \ll T^{2\eta + \eta^{1.14}} \log^6 T \ \text{ for } \ \eta < c_1,
\end{equation}
or with our new notation
\begin{equation}\label{Turan1954}
B(\eta) \leq 2 + \eta^{0.14} \ \text{ for } \eta < c_1.
\end{equation}
The next breakthrough improvement of estimates, yielding in particular DH for a fixed strip near the 1-line, was obtained in \cite{Hal-Tur-I} by Halász and Turán. In that work they combined Turán's celebrated power sum theory with the best known Korobov-Vinogradov estimates on the growth of the Riemann zeta function and Halász' pioneering idea \cite{Hal-Mittelwerte} which will play a decisive role in the main argument of this present work, too.

A special feature of our approach to the general case of Beurling systems is that we invoke our better developed knowledge about the behavior of the Riemann zeta function when addressing questions of zero density for certain Beurling systems. In particular, we will employ the best known growth bound, due to Bourgain \cite{Bourgain2017} 
\begin{equation}\label{Bourgain}
|\zeta(1/2+it)| \ll_\ve (|t|+1)^{\frac{13}{84}+\ve} \qquad  (\forall \ve >0).
\end{equation}

\section{Beurling systems and density theorems for Beurling $\zeta$ functions}\label{sec:Beurlingdensity}

Beurling systems were introduced by Beurling \cite{Beur}, and were subsequently put into a more abstract -- and hence more general -- framework, see e.g. \cite{Knopf}. Beurling systems thus can be viewed in two alternate ways, one being that they are just an arbitrary semigroup $\G$ of elements $g$, admitting an (essentially, i.e. up to ordering) unique factorization to prime elements $p \in \PP \subset \G$. In this "arithmetical semigroup" point of view the elements can be algebraic numbers, ideals of integer rings of algebraic number fields, finite Abelian groups, and many more systems which admit a suitable Krull-Schmidt type decomposition theorem leading to a multiplicative structure. In this abstract setting a key role is played by the multiplicative norming of elements, i.e, there is a norm $|\cdot| :\G \to [1,\infty)$ which is multiplicative ($|gh|=|g| |h|$) and its image is locally finite. That maps, in a  unique way, the abstract structure to $\RR_{+}$, so that one can view it as included in the real number setting where the primes $\PP$ are just an arbitrary sequence of reals, nondecreasing and tending to infinity, and freely generating the respective system of Beurling integers $\N$. Although we will  refer back to this first  interpretation here or there, the main setting of the paper is in line with the latter real number settings.

Let $\PP$ be a sequence of Beurling primes and $\N$ the generated semigroup, that is, the system of Beurling integers. Altogether, these form a Beurling number system $\B:=(\PP,\N)$. Due to its Euler product form, the Beurling zeta function does not vanish in its halfplane of convergence:
\begin{equation}\label{Euler}
\zB(s):=\sum_{n\in \N} \frac{1}{n^s} = \prod_{p \in \PP} \frac{1}{1-1/p^s} \ne 0.
\end{equation}
For a general overview of the analytic theory of Beurling number systems we refer to \cite{DZ-16}.

Beurling systems may or may not have an analytic continuation over the boundary line of convergence of their Dirichlet series expansion. However, if they do, then there are basically two possibilities: either the analytic continuation is governed by a nice asymptotic formula with power-type error for the number of integers $\N(x):=\# \{ n\le x ~:~ n \in \N\}$, or the Beurling zeta function must show extremely irregular behavior with large values and no reasonable (polynomial) bounds. This second possibility was pointed out by Frederik Broucke; for an explanation see e.g. Remark 4.1 of \cite{abc}. In turn, if $\zB$ does not admit polynomial bounds in any larger halfplane than the halfplane of convergence, then there is no hope to control its behavior and to derive finer results of zero- and hence prime distributions.

Therefore, to talk about analytic continuation of $\zB$ and to have a chance to come up with a successful analysis using it even in the critical strip, we assume that the integers are "well-behaved". That is, in this work we assume the so-called \emph{"Axiom A"} (in its normalized form to $\delta=1$) of Knopfmacher, see pages 73-79 of his fundamental book \cite{Knopf}.

\begin{definition}[Axiom A] It is said that ${\N}$ (or, loosely speaking, $\zB$)
satisfies \emph{Axiom A} -- more precisely, Axiom $A(A,\kappa,\theta)$ with the suitable constants $A, \kappa>0$ and $0 \le \theta<1$ -- if for the remainder term $\R(x):= \N(x)-\kappa x$ it holds
the estimate
\begin{equation}\label{AxiomA}
\left| \R(x) \right|  \leq A x^{\theta} \qquad (x \geq 1).
\end{equation}
\end{definition}

Historically, the first result which could be considered a weak form of a density estimate, was worked out by Kahane in \cite{K-99}. However, it was part of an indirect proof and the assumption he made was a very strong number theoretical condition. Under that restrictive assumption Kahane could prove, however, that on a given line $\Re s=a$ in the critical strip, with $a>(1+\theta)/2$, the number of $\zB$-zeros is $O(T)$.

The first zero density result for general Beurling zeta functions was obtained in \cite{Rev-D} under two extra assumptions on the Beurling system $\B$. One was that the Beurling number system remains within the realm of natural numbers -- a strong, but useful assumption, emphasized also by Knopfmacher, see pages 57--58 of \cite{Knopf}. We will term this as \emph{the integrality condition}, meaning that $|\cdot|:\G\to\NN$, that is, the norm $|g|$ of any element $g\in\G$ is a natural number. Alternatively, we may take it as $\N \subset \NN$.

Further, in \cite{Rev-D} an averaged form of the Ramanujan condition was assumed in proving the density bound
\begin{equation}\label{Mathematika}
B_\B(\eta)\le \frac{6-2\theta}{1-\theta}.
\end{equation}
In the present work we will also need a form of the Ramanujan condition, but given that our main reference \cite{Pintz2024} used it in its sharper, pointwise version, we settle with this version here, too. For its formulation, let us introduce the arithmetical function $G(\nu):= \# \{n \in \N ~:~ |n|=\nu\}$, the number of Beurling integers having a given norm (value). With that the Beurling zeta function can be written as
\begin{equation}\label{BeurlingzetawithG}
\zBs = \sum_{n\in \N} \frac{1}{n^s} = \sum_{\nu=1}^{\infty} \frac{G(\nu)}{\nu^s}.
\end{equation}
Then the Beurling system, or, loosely speaking, $\zB$ satisfies the Ramanujan condition, if $\log G(\nu) =o(\log \nu)$, that is, if for any $\delta>0$ we have $G(\nu)\le \nu^\delta$ for $\nu>\nu_0(\de)$.

The two extra conditions (integrality and the averaged Ramanujan condition) of the result \eqref{Mathematika} were removed in \cite{Rev-Arxiv} at the expense of a worse exponent. Almost simultaneously and independently, however, Frederik Broucke and Gregory Debruyne \cite{BrouckeDebruyne} succeeded in proving an improved bound also without relying on these additional assumptions. Moreover, they also showed that in the generality of assuming only Axiom A, there are Beurling zeta functions $\zB$ with $B_\B(\eta)>c_1$ with some sufficiently small positive constant $c_1$. Finally, very recently \cite{Frederikms} Frederik Broucke further improved the exponent in a general density theorem for Beurling zeta functions satisfying Axiom A, obtaining
\begin{equation}\label{Broucke2024}
B_\B(\eta) \le \frac{4}{1+2\eta-\theta}.
\end{equation}

At the end of this introduction let us point out that the new development of having zero density theorems for zeta functions of well-behaved systems of Beurling integers and primes opened up the way to achieve many strong number theoretical advances in Beurling's theory. These advances are likely to be totally impossible to reach without such tools. The interested reader may consult \cite{Rev-One, Rev-Many, BrouckeDebruyne, abc}.

\section{The aim of the paper}

Our present goal is to demonstrate that returning to the extra assumptions of integrality and the Ramanujan Condition, a new result, significantly stronger than \eqref{Broucke2024}, can be obtained. For example, if $\theta=0$, then we can almost reach the DH for small values of $\eta$, similarly to the 1954 result of Turán for the classical case. See below in Corollary \ref{Corollarythetazero}

It is of interest that while Turán used his power sum method to show \eqref{Turan1954}, here we need only Halász' idea \cite{Hal-Mittelwerte} to show a general form of \eqref{thetazero}, somewhat weaker if $\theta>0$. Namely, we will prove

\begin{theorem}\label{ThmGeneral} Under Axiom A, the integrality condition, and the Ramanujan condition  for a Beurling system $\B$, we have, as long as $0<\eta< 1-\theta$, the estimates
\begin{equation}\label{BetaforBeurling}
B_\B(\eta)=\begin{cases} \frac{2}{1-\theta-\eta} \qquad & \text{if} \quad 0< \eta < \min\left(\frac{4}{29}, \frac{8+13\theta}{71} \right)
\\
\frac{26/21}{1-4\eta} & \text{if} \quad \frac{8+13\theta}{71} \le \eta < \frac{4}{29}
\qquad ( \text{empty for} \quad \theta \ge \frac{4}{29} ),
\\  \frac{2}{1-\theta-\eta}  & \text{if} \quad \frac{4}{29}  \le \eta  < \theta \qquad ( \text{empty for} \quad \theta < \frac{4}{29} ),
\\ \frac{2}{1-2\eta} & \text{if} \quad \max\left(\theta,\frac{4}{29}\right) \le \eta
\end{cases} .
\end{equation}
\end{theorem}

We did not directly insert into the definition of the various ranges for $\eta$ the generally required condition that $\eta<1-\theta$. Note that taking into account that restriction, or even already in view of the conditions mentioned, the actual ranges for the various estimates may well be empty for specific values of the parameter $\theta$; only the very first range, -- written out fully as $0<\eta<\min\left(1-\theta,\frac{4}{29},\frac{8+13\theta}{71}\right)$ -- needs to be nonempty for all values of $\theta \in [0,1)$.

Specializing for $\theta=0$ and for small $\eta$ only, we obtain

\begin{corollary}\label{Corollarythetazero} Let $\B$ be a Beurling number system satisfying the integrality condition and also Axiom A with $\theta=0$. Then for the Beurling zeta function $\zB$ of this system we have a zero density bound with
\begin{equation}\label{thetazero}
B_\B(\eta)=\frac{2}{1-\eta} \left( = 2+O(\eta)\right) \qquad \textrm{for} \quad \eta \le \frac{8}{71}.
\end{equation}
\end{corollary}

\section{The main tool}\label{sec:maintool}

Our main tool is a slightly strengthened variant of a general zero density type theorem proved recently by the first named author in \cite{Pintz2024}. This deals with general (but ordinary, i.e. with integer powers in the denominator) Dirichlet series.

Assuming $f_1 \neq 0$, we consider the (Dirichlet inverse to each other) arithmetical functions $f_n$ and $g_n$, both satisfying the Ramanujan condition, and the corresponding reciprocal pair of Dirichlet series
\begin{equation}
\label{eq:2.1}
f(s) = \sum_{n = 1}^\infty \frac{f_n}{n^s}, \quad M(s) = \frac1{f(s)} = \sum_{n = 1}^\infty \frac{g_n}{n^s} ,
\end{equation}
which by assumption about the Ramanujan condition must be analytic for $\sigma > 1$, and satisfy $M(s)f(s) = 1$ there.

\begin{remark}\label{rem:1}
If $f_n$ is completely multiplicative as a function of $n$ then $f_1 = 1$ and $g(n) = \mu(n)f_n$.
\end{remark}

Further we suppose that with some fixed constant $\alpha_f < 1$, the function $f(s)$ can be continued analytically to the halfplane $\sigma > \alpha_f$, save a simple pole at $s = 1$ with residue $f_0$. We define the respective "generalized Lindelöf function" to any such $f$ as follows.
\begin{equation}
\label{eq:2.3}
\mu_f(\sigma_0) := \text{\rm inf} \left\{ \mu; |f(\sigma + it)| \leq T^\mu \text{ for } \sigma \geq \sigma_0, 1 \leq |t| \leq T\right\} < \infty \text{ for } \sigma_0 > \alpha_f.
\end{equation}
This is clearly Lindel\"of's classical $\mu$-function if $f(s)$ is chosen to be $\zeta(s)$.
In the following technical definition the function $\lambda_f^{(0)}$ will depend on $\mu_f$; in particular, $\lambda_\zeta^{(0)}$ on $\mu_\zeta$.
Let
\begin{equation}
\label{eq:2.4}
\lambda_f^{(0)}(\eta)\! := \!\! \inf_{0<a; (a + 1)\eta < 1 - \alpha_f} \frac{\mu_f(1\! -\!
(a\! +\! 1)\eta)}{a\eta},
\quad \lambda_\zeta^{(0)}(\eta) \! := \!\! \inf_{0 <b} \frac{\mu_\zeta(1\! -\!
(b\! +\! 1)\eta)}{b\eta}.
\end{equation}
Note that we do not assume anything about the size of the parameter $b$ in the above definition for $\lambda_\zeta^{(0)}(\eta)$; this is the essential change compared to the original version of the main auxiliary theorem what we want to use.

In the following we do not need \emph{the exact} values of these Lindelöf-type derived functions, (which is rather fortunate, given that not even the precise value of $\mu_\zeta$ is known), but we will be satisfied with any function $\lambda_f(\eta) \ge \lambda^{(0)}_f(\eta)$, in particular with $\lambda_\zeta(\eta) \ge \lambda^{(0)}_\zeta(\eta)$. Even if precise values are not known, strong estimates on the growth of the Riemann zeta function $\zeta$ will furnish us reasonably strong bounds $\lambda_\zeta(\eta)$.

As is usual, we denote the number of zeroes (counting them with their possible multiplicities) as
\begin{equation}\label{Zeronumbersforf}
N_f(1-\eta,T):= \#\{\rho~:~ f(\rho)=0, \Re \rho \in [1-\eta,1], |\Im \rho|\le T\}.
\end{equation}
The general density estimate, what we apply in this work, reads as follows.

\begin{lemma}\label{Thm:Pintzdensity} Let $f$ and $g$ be reciprocal Dirichlet series satisfying the above assumptions and the Ramanujan condition $\log|f_n|=o(\log n)$ and $\log|g_n|=o(\log n)$.

Let $0<\eta<1-\alpha_f$ be arbitrary, and let $\lambda_f(\eta)$ and $\lz(\eta)$ be some estimator functions satisfying $\lambda_f(\eta) \ge \lambda_f^{(0)}(\eta)$ and $\lz(\eta) \ge \lz^{(0)}(\eta)$, where $\lambda_f^{(0)}(\eta), ~\lz^{(0)}(\eta)$ are defined according to \eqref{eq:2.3} and \eqref{eq:2.4}.
Denote
\begin{equation}\label{Betadef}
B_f(\eta):= \max\left( 2\lambda_f(\eta), 4\lz(2\eta) \right) .
\end{equation}

Then if the condition
\begin{equation}\label{crucialcondition}
\lambda_f(\eta) \ge \lambda_\zeta(2\eta)
\end{equation}
holds true, then we have the zero density estimate
\begin{equation}\label{Pintzdensity}
N_f(1-\eta,T) \ll_{\eta,\ve} T^{B_f(\eta)\eta+\ve}.
\end{equation}
\end{lemma}

\begin{proof}
This is essentially a part of Theorem 2 in \cite{Pintz2024}, listed under formula (2.9) there.

However, here we formulated  an alteration in that we changed also the definition of $\lambda_\zeta^{(0)}(\eta)$ somewhat, allowing the $b$-parameter, originally restricted to satisfy some upper bound in \cite{Pintz2024}, become unbounded here.

It is clear from the original proof that
\begin{itemize}
  \item[(i)] we can allow $\lambda_f(\eta)=\lambda_\zeta(2\eta)$, too (no need for strict inequality, as in the original formulation);
  \item[(ii)] we can allow in the definition of the original $\lambda_f(\eta)$ function (which corresponds to $\lambda_f^{(0)}(\eta)$ here) any upper estimation, as we formally described above in detail;
  \item[(iii)] when defining $\lambda_\zeta^{(0)}(\eta)$, actually there is no need for any upper bound for $b$, as the original proof works without change even for larger values of $b$, given the analytic continuation (and known bounds of the growth) of the Riemann $\zeta$ function on the whole complex plane.
\end{itemize}

With these alterations we obtain a proof of the Lemma.
\end{proof}

To obtain somewhat sharper results for larger values of $\eta$ in our present work, we needed this change for two reasons. First, the possibility of allowing $\eta$ to extend up to the theoretical bound of $1-\theta$, we need good estimates on $\lambda_\zeta(2\eta)$ in order to meet the condition \eqref{crucialcondition} for $f=\zB$. Without this amelioration, the validity of the condition would not be ascertained for $\eta$ exceeding certain bounds. Second, with the above amelioration, also the bound \eqref{Betadef} improves, providing somewhat better density results for larger values of $\eta$.


\section{The proof of Theorem \ref{ThmGeneral}}\label{sec:proofThm}

\subsection{Computing estimates for the estimator $\lz(2\eta)$ of the Riemann $\zeta$ function}\label{sec:lambdazeta}

We start with finding a sufficiently good estimator function $\lz$ (which we will later use at the value of $2\eta$, not at $\eta$). Note that here we deal with the Riemann $\zeta$ function, so that there is no upper bound on the admissible values of $\eta$; however, as for the Beurling zeta function we have the natural bound $0<\eta<1-\theta\le 1$, we do not extend over $0<\eta<1$ (or, if considering values at $2\eta$, we restrict equivalently for $0<\eta<1/2$).

We work with two different choices of the appropriate parameter $b$, depending on the size of $\eta$. More precisely, for $\eta\le 4/29$ we will use Bourgain's estimate \eqref{Bourgain} with choosing $b:=\frac{1}{4\eta}-1 \Longleftrightarrow (b+1)2\eta=1/2$, which leads to $\mu_\zeta(1-(b+1)2\eta))= \mu_\zeta(1/2)\le 13/84$ and hence the estimator function
\begin{equation}\label{lambdazetaforsmalleta}
\lz(2\eta)=\frac{13/84}{1/2-2\eta}=\frac{13/42}{1-4\eta},
\end{equation}
obviously satisfying $\lz(2\eta)\ge \lz^{(0)}(2\eta)$ for this range of values of $\eta$.

If $4/29 \le \eta \le 1/4$, then we will use the basic estimate $\mu_\zeta(0)=1/2$. Choosing the value $b:=\frac{1}{2\ee}-1 \Longleftrightarrow (b+1)2\eta=1$ in this case, we obtain from \eqref{eq:2.4} that we can take
\beq\label{lambdazeta2eta}
\lz(2\eta)=\frac{1/2}{1-2\eta}.
\eeq


If $\eta>1/4$ then $\lambda_\zeta (2\eta)>1$, hence by \eqref{Betadef} $B_\B(\eta) \ge 4 \lz (2\eta)>1/\eta$, which is weaker then the trivial estimate $1$, see Section \ref{sec:conclusion}.

\medskip
\subsection{Estimator for the Beurling zeta function $\zB$}

For the Beurling zeta function of our Beurling system $\B$, we cannot take values with $\Re s\le \theta$, so that we must restrict to $\eta<1-\theta$, and in the construction of $\lB^{(0)}(\eta)$ and $\lB(\eta)$ even to $1-(a+1)\eta >\theta$, i.e. $(a+1)\eta<1-\theta$.

We will extend $(a+1)\eta$ close to this limit. A reference to Lemma 5 of \cite{Rev-Arxiv} (or, with full proof, to Lemma 2.5 of \cite{Rev-MP}) furnishes
\beq\label{mubetaepsilontheta}
\mu_{\zB}(\theta+\ve)\le \frac{1-\theta+\ve}{1-\theta},
\eeq
which in turn allows to set $a:=\frac{1-\theta-\ve}{\eta}-1 \Longleftrightarrow (a+1)\eta=1-\theta-\ve$ and derive with this value the estimator functions
$$
\lB^{(\ve)}(\eta) = \frac{1-\theta+\ve}{(1-\theta)(1-\theta-\eta-\ve)}.
$$
As $\lB^{(\ve)}(\eta) \ge \lB^{(0)}(\eta)$ holds for all $\ve>0$, we in fact can take limits with respect to $\ve\to 0$ and obtain the new estimator function
\beq\label{lambdaetatheta}
\lB(\eta) = \frac{1}{1-\theta-\eta}.
\eeq

Unlike the estimator functions for the Riemann $\zeta$ function, where we distinguished two different cases with two different estimates \eqref{lambdazetaforsmalleta} and \eqref{lambdazeta2eta},
the estimate \eqref{lambdaetatheta} is "universal" (the same formula) for all admissible values of $0<\eta<1-\theta$.

\medskip
\subsection{Verification of the conditions of the key lemma}

Next we check the conditions of the key lemma in order to apply it for admissible values of $\eta$. We always assume that $\eta<1-\theta$ without repeatedly telling about it.

For the small values $0<\eta\le 4/29$, an easy calculation furnishes that $\lB(\eta)=\frac{1}{1-\theta-\eta} \ge \frac{1}{1-\eta} \ge \frac{13/42}{1-4\eta} =\lz(2\eta)$, the condition in \eqref{crucialcondition}. Let now $4/29 \le \eta \le 1/4$. Then we find $\lB(\eta)=\frac{1}{1-\theta-\eta}\ge \frac{1}{1-\eta} \ge \frac{1}{2-4\eta}=\lz(2\eta)$, that is, we again have condition \eqref{crucialcondition} of Lemma \ref{Thm:Pintzdensity} satisfied.

In all, we can record that with the estimator functions worked out above for the Riemann $\zeta$ and the Beurling zeta $\zB$, the last condition  \eqref{crucialcondition} of Lemma \ref{Thm:Pintzdensity}, i.e. the inequality  $\lB(\eta) \ge \lz(2\eta)$, holds true for the entire domain $0<\eta<1-\theta$.

\medskip
\subsection{End of the proof of Theorem \ref{ThmGeneral}}
Therefore, for all values of $0<\eta<1-\theta$, we can apply Theorem A, which yields
\beq\label{almostfinalformsalleta}
B_{\B}(\eta)= \max\left( 2\lB(\eta), 4\lz(2\eta) \right)=\begin{cases}
\max\left(\frac{2}{1-\theta-\eta}, \frac{26/21}{1-4\eta} \right) \qquad &\text{if}\quad \eta\le \min(4/29, 1-\theta) \\
\max\left(\frac{2}{1-\theta-\eta}, \frac{2}{1-2\eta} \right) \qquad &\text{if}\quad 4/29 \le \eta \le \min(1/4, 1-\theta)
\end{cases}.
\eeq

In case of $\eta \le 4/29$, the first term is maximal precisely when $\eta \le (8+13\theta)/71$.

In case of $4/29 \le \eta \le 1/4$, the first term gives the maximum if and only if $\eta \le \theta$. 

Winding up these partial case calculations furnishes \eqref{BetaforBeurling}.

$\square$

\section{Conclusion}\label{sec:conclusion}

Contrary to zeta functions with a functional equation, 
there is no asymptotic formula for the number of zeroes of a Beurling zeta function in the critical strip or in a halfplane strictly in the critical strip. This is not just a weakness of methods for a proof; Beurling systems with only finitely many zeroes (or no zeroes at all) do exist, see \cite{Rev-One}, Theorem 7.4. (though one may recall that these phenomenon is known only for $\theta \ge 1/2$ and strips $\Re s\ge \sigma >\theta$). Nevertheless, an upper estimate of the same order of $T\log T$ as in the Riemann zeta case, is known, see e.g. (7) of Theorem 2 from \cite{DMV} or Lemma 3.5 of \cite{Rev-MP}. Therefore, a density bound exceeding $T^{1+\ve}$ does not provide anything new. In light of this the part for $\eta \ge \frac{1-\theta}{3}$ of the above results become worthless: they can be substituted by the better $B_\B^{\star}=1/\eta$ (equivalent to $N(\sigma,T)\ \ll T^{1+\ve}$ for all $\ve>0$). Also, one may note that our above obtained bounds become inferior to the recent results of Broucke \eqref{Broucke2024} from \cite{Frederikms} for $\eta>\frac{1-\theta}{4}$. However, for $\eta<\frac{1-\theta}{4}$, our results provide sharper estimates, even though only under the two assumptions on the integrality and the Ramaujan condition.

\bigskip

\noindent
\hspace*{5mm}
\begin{minipage}{\textwidth}
\noindent
\hspace*{-5mm}
János {} Pintz\\
HUN-REN Alfréd Rényi Institute of Mathematics\\
Reáltanoda utca 13-15\\
1053 Budapest, Hungary \\
{\tt pintz@renyi.hu}
\end{minipage}

\bigskip
\bigskip

\noindent
\hspace*{5mm}
\begin{minipage}{\textwidth}
\noindent
\hspace*{-5mm}
Szilárd Gy.{} Révész\\
HUN-REN Alfréd Rényi Institute of Mathematics\\
Reáltanoda utca 13-15\\
1053 Budapest, Hungary \\
{\tt revesz.szilard@renyi.hu}
\end{minipage}

\end{document}